\documentclass{amsart}
\usepackage[cp1250]{inputenc}
\usepackage{picture}
\usepackage{amssymb}
\usepackage{color}
\newtheorem{theorem}{Theorem}[section]

\newtheorem{proposition}[theorem]{Proposition}
\newtheorem{problem}[theorem]{Problem}
\newtheorem{corollary}[theorem]{Corollary}

\newtheorem{example}[theorem]{Example}
\newtheorem{definition}[theorem]{Definition}
\newtheorem{remark}[theorem]{Remark}

\newcommand{\N}{\mathbb N}

\newcommand{\R}{\mathbb R}

\author{Jacek Marchwicki}
\address{Institute of Mathematics, \L \'od\'z University of Technology,
W\'olcza\'nska 215, 93-005 \L \'od\'z, Poland}
\email {marchewajaclaw@gmail.com}

\title[Achievement sets and sum ranges with ideal supports]{Achievement sets and sum ranges with ideal supports}


\begin{document} 

\begin{abstract}

We introduce the notion of ideally supported achievement sets for a series of real numbers. We analize their complexity and topological properties. We compare the notion of ideal achievement sets with the notion of ideally supported sum range of real series, considered by Filipów and Szuca. We complete Filipów and Szuca characterization of ideal sum ranges, [R. Filipów, P. Szuca, Rearrangement of conditionally convergent series on a small set, J. Math. Anal. Appl. 362 (2010), no. 1, 64-71.], and we obtain some generalization of Riemann's Theorem.

\end{abstract}
\maketitle

\section{Introduction}
By the achievement set of a series $\sum_{n=1}^{\infty}x_n$ we mean the set $A(x_n)=\{\sum_{n\in A} x_n : A\subset\mathbb{N}\}$. 
By $SR(x_n)=\{\sum_{n=1}^{\infty}x_{\sigma(n)} : \sigma\in S_{\infty}\}$ we denote the set of all convergent rearrangements $\sum_{n=1}^{\infty}x_{\sigma(n)}$ of $\sum_{n=1}^{\infty}x_{n}$, that is the sum range of the series $\sum_{n=1}^{\infty}x_{n}$. Kakeya in \cite{Kakeya} proved that if a series $\sum_{n=1}^{\infty}x_n$ of reals is absolutely convergent and contains infinite many non-zero terms, then
\begin{itemize}
\item $A(x_n)$ is a compact perfect set;
\item $A(x_n)$ is homeomorphic to the Cantor set for quickly convergent series  $\sum_{n=1}^{\infty}x_n$, that is if $\vert x_n\vert>\sum_{k=n+1}^{\infty}\vert x_k\vert$ for every $n\in\mathbb{N}$;
\item  $A(x_n)$ is a finite sum of closed intervals, if $\vert x_n\vert\leq\sum_{k=n+1}^{\infty}\vert x_k\vert$ for almost all $n\in\mathbb{N}$.
\end{itemize}
The full topological characterization is due to Guthrie, Nymann and Saenz \cite{GN,NS1}, who showed that the achievement set of an absolutely summable sequence of reals is one of the following form: a finite set, a finite union of intervals, a  homeomorphic copy of the Cantor set or a Cantorval, which is a set homeomorphic to the union of the Cantor set and sets which are removed from the unit segment by even steps of the Cantor set construction. This characterization is not valid for series of complex numbers and for multidimensional series, see \cite{BG}. 

Theory of achievement sets for absolutely convergent series is equivalent to the that of ranges of finite purely atomic measures. If $\mu$ is finite and purely atomic on a set $X$, then there is a countable set $S=\{a_1,a_2,\ldots\}$ such that $\mu(S)=\sum_{n=1}^{\infty} \mu(\{a_n\})=\mu(X)$, where $a_n$ is an atom of $\mu$. Let $x_n=\mu(\{a_n\})$, then $\mu(A)=\sum_{n\in E}x_n$, where $E=\{n : a_n\in A\}$. Thus $range(\mu)=\{\mu(A): A \ \text{is a measurable subset of} \ X\}=\{\sum_{n\in E}x_n : E\subset\mathbb{N}\}=A(x_n)$.

During last decades many authors have defined ideal versions of important Analysis notions and proved many remarkable results. Since the convergence is a basic notion in Analysis, most of them deal with  ideal convergence of sequences \cite{BGW},\cite{KSW},\cite{S}. The following list of topics and related papers is far from being complete and it gives only a flavor of these matters: ideal convergence of sequences of functions \cite{BDK}; ideal convergence of series \cite{GO},\cite{LO}; ideal convergence in measure \cite{K},\cite{M}; ideal versions of combinatorial theorems \cite{FMRS}; ideal versions of the Riemann rearrangement theorem and the Levy-Steinitz theorem \cite{FS},\cite{K}; ideal version of the Banach principle \cite{HS}.


We define the ideal achievement set in a natural way, namely $A_{I}(x_n)=\{\sum_{n\in A} x_n : A\in I\}$. This is a subset of $A(x_n)$. We study how properties of $A_{I}(x_n)=\{\sum_{n\in A} x_n : A\in I\}$ and its possible form depend on the  properties of a given ideal $I$ and a sequence $(x_n)$. Note that for a sequence $(x_n)\in \ell_1$ and an ideal $I\supseteq Fin$, $I\neq Fin$ we have $A(x_n)=A_{I}(y_n)$, provided $(y_n)\in \ell_1$ is defined as follows: $y_{b_{n}}=x_{n}$ for $n\in\mathbb{N}$ and $y_{k}=0$ for $k\notin B$, where $B=\{b_{n}\}_{n=1}^{\infty}$ is an infinite set from $I$. If we consider  $\ell_1^*=\{(x_n)\in \ell_{1} : x_n\neq 0 \ \text{for \ each} \ n\in\mathbb{N}\}$, then the theory of ideal achievement sets differs from the theory of standard achievement sets, in particular for a non-maximal ideal $I$ one can construct a sequence $(x_n)\in \ell_{1}^*$ for which the set $A_{I}(x_n)$ is open, see \ref{idealmaksymalnyizbiorotwarty}.

The paper is organized as follows: 
\\In Section 2 we give background definitions and facts, which we use in next sections. In Section 3 we consider conditionally convergent series and divergent series. We prove the generalization of the Riemann's Rearangement Theorem  and show that ideally supported sum range of any conditionally convergent series is one of the following: a point, a real line or a set, which contains a halfline and in some particular cases it is exactly a closed halfline. We also prove that for any conditionally convergent series $\sum_{n=1}^{\infty}x_n$ there exists an ideal $I$ such that  $A_{I}(x_n)=\mathbb{R}$ and the stronger condition, defined in Theorem \ref{filipowszuca} (ii) - Filipów and Szuca characterization of ideals for which thesis of Riemann's Theorem holds - is not satisfied. In Section 4 we study absolutely convergent series. We give many examples of ideally supported achievement sets with high Borel complexity. We also show that it can be a set, which is not measurable. Moreover, we prove that for any ideal $I\supsetneq Fin$ we may construct a series  $\sum_{n=1}^{\infty}x_n$ such that  $A_{I}(x_n)$ is equal to  $A(x_n)$ up to a point and if $I$ is not maximal, then $A_{I}(x_n)$ can be an open set. In Section 5 we show how we can modify $A_{I}(x_n)$ to remain symmetry of $A(x_n)$. Section 6 is dedicated to examples, which show inclusions between ideally supported achievement sets and in Section 7 we give some open problems.

\section{Background}
We use standard set-theoretic notation, \cite{Kechris}.
We say that $I\subset P(\mathbb{N})$ is an ideal if for every $A,B\in I$ we have $A\cup B\in I$ and for every $A\in I$ and every $B\subset A$ we have $B\in I$, moreover $\mathbb{N}\notin I$. By $Fin$ we denote the set $\{A\subset\mathbb{N} : \vert A\vert<\infty\}$ of all finite subsets of $\mathbb{N}$ which is clearly an ideal. In the sequel, we will consider ideals $I$, which contain $Fin$, symbolically $Fin\subset I$. Put $I_{(a_n)}=\{A\subset\mathbb{N} : \sum_{n\in A}a_n \ \text{converges} \}$,  where $\sum_{n=1}^{\infty}a_n$ is a given divergent series of positive terms; such family of sets forms a so-called summable ideal. We say that an ideal $I$ is dense if for every set $A$ with $\vert A\vert=\infty$ there exists $B\subset A$ such that $\vert B\vert =\infty$ and $B\in I$. An ideal $I$ is not maximal if there exists an ideal $J$ such that $I\subsetneq J$. Otherwise we say that $I$ is maximal. It is well known that $I$ is maximal if and only if for each $A\subset\mathbb{N}$ either $A\in I$ or $\mathbb{N}\setminus A\in I$. By $A+Fin$ $((k_n)+Fin)$ we denote the smallest ideal generated by $A$ and $Fin$ (by $\{k_n : n\in\mathbb{N}\}$ and $Fin$, respectively).

Let us consider the function $f: \{0,1\}^{\mathbb{N}}\rightarrow\mathbb{R}$ defined as $f(\chi_A)=\sum_{n\in A}x_n$, we will call it an associated function of the series $\sum_{n=1}^{\infty} x_n$.  We equip $\{0,1\}^{\mathbb{N}}$ with Tichonov's topology, that is a topology given by sub-base $\{\{0,1\}^{k-1}\times\{i\}\times\{0,1\}^{\mathbb{N}}, k\in\mathbb{N}, i\in\{0,1\}\}$. Identifying sets $A\subset\mathbb{N}$ with their characteristic functions $\chi_A$, we may consider on $P(\mathbb{N})$ the topology inherited from $\{0,1\}^{\mathbb{N}}$. Therefore we may consider topological properties of ideals $I\subset P(\mathbb{N})$. We say that $I$ is Borel ($F_{\sigma}$, $F_{\sigma\delta}$, of the Baire property, measurable) provided $\{\chi_A : A\in I\}$ is Borel ($F_{\sigma}$, etc. respectively) in $\{0,1\}^{\mathbb{N}}$.
On the real line we consider the natural topology. If $(x_n)\in \ell_1$, then $f$ is continuous. Moreover if $f$ is one-to-one (for example if the series $\sum_{n=1}^{\infty}x_n$ is quickly convergent), then $f$ is a homeomorphism between $\{0,1\}^{\mathbb{N}}$ and $A(x_n)$. Hence $A_I(x_n)=f(I)$ and $f^{-1}(A_I(x_n))=I$. Since homeomorphic pre-images of Borel $(\Sigma_{\alpha}^{0},\Pi_{\alpha}^{0})$ sets are Borel $(\Sigma_{\alpha}^{0},\Pi_{\alpha}^{0})$, then the descriptive complexity of $A$ and $f^{-1}(A)$ is equal provided $f$ is a homeomorphism.

Let $\mu_n(\{0\})=\mu_n(\{1\})=\frac{1}{2}$ for each $n\in\mathbb{N}$. We consider the product measure $\mu=\prod_{n=1}^{\infty} \mu_n$ on $\{0,1\}^{\mathbb{N}}$ and the function $f: \{0,1\}^{\mathbb{N}}\supset \chi_A\rightarrow \sum_{n\in A}\frac{1}{2^n}$.  By $\lambda$ we denote the Lebesgue measure on $[0,1]$. Then 
\begin{itemize}
\item If $X\subset\{0,1\}^{\mathbb{N}}$ is measurable, then $f(X)$ is measurable on $[0,1]$ and $\mu(X)=\lambda(f(X))$, i. e. $f$ preserves Lebesque measure; 
\item If $X$ is non-measurable, then $f(X)$ is also non-$\lambda$-measurable;
\item If $X$ is meager, then so is $f(X)$.
\end{itemize}
 We say that $F\subset P(\mathbb{N})$ is a filter if for every $A,B\in F$ we have $A\cap B\in F$ and for every $A\in F$ and every $B\supset A$ we have $B\in F$, moreover $\emptyset\notin F$. For an ideal $I$ we consider its dual filter $F_I=\{A : \mathbb{N}\setminus A\in I\}$. We also consider $A_{F_I}(x_n)=\{\sum_{n\in A}x_n : A\in F_I\}$.
It is well-known that if $I$ is an ideal ($F$ is a filter), then $I$ ($F$) has a Baire property if and only if  $I$ ($F$) is meager. Similarly $I$ ($F$) is measurable if and only if  $I$ ($F$) is null, \cite{T}. From this we easily obtain that maximal ideals neither satisfy Baire property nor are Lebesgue measurable. Indeed, if $I$ is maximal, then its complement equals to its dual filter $F_I$, which is maximal as well. If $I$ would have Baire property, then $I$ would be meager. Since $\chi_A\mapsto\chi_{\N\setminus A}$ is a homeomorphism of $\{0,1\}^\mathbb{N}$, then $F_I$ is also meager, and we reach a contradiction. The argument for measure case is the same -- we use the fact that $\chi_A\mapsto\chi_{\N\setminus A}$ preserves the measure $\mu$ on $\{0,1\}^\mathbb{N}$. 


\begin{definition}
We say that  $\phi :P(\mathbb{N})\rightarrow [0,\infty]$ is a submeasure iff
\begin{itemize}
\item $\phi(\emptyset)=0$;
\item $\phi(A\cup B)\leq\phi(A)+\phi(B)$ for each $A,B\subset\mathbb{N}$;
\item $\phi(A)\leq\phi(B)$ for each $A\subset B\subset\mathbb{N}$.
\end{itemize}
Moreover if $\phi(A)=\lim_{n\to\infty}\phi(A\cap\{1,\ldots,n\})$  for every $A\subset\mathbb{N}$ then we say that $\phi$ is upper-semicontinuous.
\end{definition}
There is a nice characterization of $F_{\sigma}$-ideals by Mazur, in terms of submeasures, \cite{Mazur}: 
\begin{theorem}[Mazur]
An ideal $I$ is $F_{\sigma}$ if and only if there exists upper-semicontinuous submeasure $\phi :P(\mathbb{N})\rightarrow [0,\infty]$ such that $I=Fin(\phi)=\{A\subset\mathbb{N} : \phi(A)<\infty\}$.
\end{theorem}
Using Mazur's characterization we can simply show that $Fin$ and any summable ideal $I_{(a_n)}$ are $F_{\sigma}$-ideals. Indeed, $Fin=Fin(\phi)$ for $\phi(A)=\vert A\vert$ and  $I_{(a_n)}=Fin(\phi)$ for $\phi(A)=\sum_{n\in A}a_n$. 
The following result is folklor but we present its short proof.
\begin{proposition}
Ideal $I\supset Fin$ is not a $G_{\delta}$-set. In particular $Fin$ is not $G_\delta$.  
\end{proposition}
\begin{proof}
Note that $Fin$ and $F_{Fin}$ are dense in $\{0,1\}^{\mathbb{N}}$. The space $\{0,1\}^{\mathbb{N}}$ is compact, which implies its completeness. Since $I\supset Fin$, we know that $I$ is dense. Suppose that $I$ is $G_{\delta}$. Since $\chi_A\rightarrow\chi_{\mathbb{N}\setminus A}$ is a homeomorphism of $\{0,1\}^{\mathbb{N}}$ onto itself we obtain that $F_I$ is also $G_{\delta}$ and since $F_I\supset F_{Fin}$ we obtain its density. Furthermore $I\cap F_I=\emptyset$, which by Baire's theorem leads to a contradiction.
\end{proof}

\section{Conditionally convergent series of reals}
Let $\sum_{n=1}^{\infty}x_n$ be a conditionally convergent series of reals. By the Riemann Theorem we know that $SR(x_n)=\mathbb{R}$ and it is also known that $A(x_n)=\mathbb{R}$, see \cite{J}. The set $A_{Fin}(x_n)$ is dense on the real line, because every sum of the series can be approximated by its partial sums. Since $I\supset Fin$, we get  $\overline{A_{I}(x_n)}=\mathbb{R}$. 

To our best knowledge this is a first paper in which achievement set is considered with respect to an ideal, although ideal-sum ranges have been considered before. In \cite{FS} Filipów and Szuca defined an ideally supported sum range $SR_{I}(x_n)=\{\sum_{n=1}^{\infty}x_{\sigma(n)} : \sigma\in S_{\infty}, supp(\sigma)=\{n : \sigma(n)\neq n\}\in I\}$ for an ideal $I$. Filipów and Szuca were interested whether $SR_{I}(x_n)=\mathbb{R}$ for any conditionally convergent series $\sum_{n=1}^{\infty}x_n$. They characterized ideals $I$ with this property, where a crucial role was played by summable ideals.
Filipów and Szuca's characterization reads as follows:
\begin{theorem}\label{filipowszuca}
The following assertions are equivalent:\\
(i) ideal $I$ is not contained in any summable ideal;\\
(ii) for every conditionally convergent series $\sum_{n=1}^{\infty}x_n$ there exists $W\in I$ such that the series $\sum_{n\in W}x_n$ is conditionally convergent;\\
(iii) for every conditionally convergent series $\sum_{n=1}^{\infty}x_n$ we have $\{\sum_{n=1}^{\infty}x_{\sigma(n)} : supp(\sigma)\in I\}=\mathbb{R}$.\\
\end{theorem}
In our notation (iii) can be written as $SR_{I}(x_n)=\mathbb{R}$. From (ii) we immediately obtain  $A_{I}(x_n)=\mathbb{R}$. However, the equality $A_{I}(x_n)=\mathbb{R}$ does not imply (ii) in general, which is shown in the following examples.
\begin{example}\label{niegestyideal}
\emph{
Define $x_{2n-1}=\frac{(-1)^n}{n}, x_{2n}=\frac{1}{2^n}$. Let us consider $I=2\mathbb{N}+Fin$. Note that the series  $\sum_{n=1}^{\infty}x_n$ is condtionally convergent and for each $A\cup F\in I$, where $A\subset 2\mathbb{N}$ and $F\in Fin$, we have  $\sum_{n\in A\cup F}\vert x_n\vert=\sum_{n\in A}\vert x_n \vert+\sum_{n\in F\setminus A}\vert x_n \vert\leq 1+\sum_{n\in F\setminus A}\vert x_n \vert<\infty$, since $F\in Fin$. Hence $\sum_{n\in A\cup F}x_n$ is absolutely convergent, so it cannot be conditionally convergent. In particular, it implies that $SR_{I}(x_n)=\{\sum_{n=1}^{\infty}x_n\}$, since a rearrangement of absolutely convergent series does not change its limit.
\\On the other hand, for each $x\in\mathbb{R}$ one can find a finite set $G\subset 2\mathbb{N}-1$ such that $x-\sum_{n\in G} x_n=y\in[0,1]$. Let $B\subset 2\mathbb{N}$ be such that $y=\sum_{n\in B} x_n$. Thus $B\cup G\in I$ and $x=\sum_{n\in B\cup G} x_n$. Hence $A_{I}(x_n)=\mathbb{R}$. 
}
\end{example}
\begin{example}
\emph{
Let $x_n=\frac{(-1)^n}{n}$ for $n\in\mathbb{N}$ and $I=I_{(\frac{1}{n})}$. Then $A_{I}(x_n)=\mathbb{R}$.
Indeed, fix $x>0$. Since $x_{2n}\to 0$ and $\sum_{n=1}^{\infty}x_{2n}=\infty$, then there exists $F\subset 2\mathbb{N}$ such that $\sum_{n\in F}x_n=x$. Clearly $F\in I$. For $x<0$ we take $F\subset 2\mathbb{N}-1$. 
\\Suppose that there exists $W\in I$ such that  $\sum_{n\in W}x_{n}$ is conditionally convergent. Then  $\sum_{n\in W\cap 2\mathbb{N}}x_{n}=\infty$ and $\sum_{n\in W\cap 2\mathbb{N}-1}x_{n}=-\infty$.
Hence $\sum_{n\in W}\frac{1}{n}=\sum_{n\in W}\vert x_n\vert=\infty$, which means that $W\notin I$.
Finally, note that $I$ is a dense ideal, while that defined in Example \ref{niegestyideal} is not dense. Moreover, $I$ is a summable ideal and $SR_{I}(x_n)=\{\sum_{n=1}^{\infty}x_n\}$.
}
\end{example}
Now we will show that for every conditionally convergent series $\sum_{n=1}^{\infty}x_n$ we can construct ideal $I$ with $A_{I}(x_n)=\mathbb{R}$ and Theorem \ref{filipowszuca}(ii) is not fulfilled.
\begin{theorem}
Let  $\sum_{n=1}^{\infty}x_n$  be a conditionally convergent series. Then for any $a,b\in\mathbb{R}$, $a<b$ there exists its absolutely convergent subsequence $\sum_{n=1}^{\infty}x_{k_n}$ such that $A(x_{k_n})$ contains the interval $[a,b]$.
\begin{proof}
Let $x=b-a$. Then one can find a sequence of finite sets of indices $(F_n)_{n=1}^{\infty}$ such that 
\begin{enumerate}
\item $\max F_n<\min F_{n+1}$ for every $n\in\mathbb{N}$;
\item for each $j\in\cup_{n=1}^{\infty}F_n$ we have $x_j>0$; 
\item $\frac{y_{n-1}}{2^n}\leq y_n\leq \frac{3x}{2^{n+1}}$, where $y_0=x$, $y_n=\sum_{k\in F_n}x_k$ for each $n\in\mathbb{N}$; 
\end{enumerate} 
Note that by $(1)$ we get that $A((x_n)_{n\in \cup_{k=1}^{\infty}F_k})\subset A(x_n)$ and by $(2)$ and $(3)$ we obtain that $\sum_{n\in \cup_{k=1}^{\infty}F_k}x_n$ is absolutely convergent. Fix $y\in[0,x]$. Then $y=\sum_{n=1}^{\infty}\varepsilon_n y_n$, where we define $(\varepsilon_n)_{n=1}^{\infty}$ inductively in the following way $\varepsilon_n=1$ if $y-\sum_{k=1}^{n-1}\varepsilon_k y_k\geq y_n$ and $\varepsilon_n=0$ otherwise. Hence $A((x_n)_{n\in \cup_{k=1}^{\infty}F_k})\supset A(y_n)\supset [0,x]$. One can find an absolutely convergent subseries $\sum_{n=1}^{\infty}x_{p_n}=a$ with each term not greater than $x$ and such that $(p_n)\cap \bigcup_{k=1}^{\infty}F_k=\emptyset$. Hence $(k_n)=(p_n)\cup \bigcup_{k=1}^{\infty}F_k=\emptyset$ satisfies the thesis.
\end{proof}
\end{theorem}
\begin{corollary}
Let $\sum_{n=1}^{\infty}x_n$  be a conditionally convergent series, then there exists an ideal $I$ such that  $A_{I}(x_n)=\mathbb{R}$ and the assertion of Theorem \ref{filipowszuca}(ii) is not satisfied. 
\begin{proof}
Let $\sum_{n=1}^{\infty}x_{k_n}$ be an absolutely convergent subsequence such that  $A(x_{k_n})\supset [a,b]$. Define $I=(k_n)+Fin$. Note that $\sum_{n\in\mathbb{N}\setminus (k_n)}x_n$ is conditionally convergent, so $A((x_n)_{n\in\mathbb{N}\setminus (k_n)})=\mathbb{R}$ and what is more $\overline{A}_{Fin}((x_n)_{n\in\mathbb{N}\setminus (k_n)})=\mathbb{R}$. Fix $x\in\mathbb{R}$. On can find $F\subset \mathbb{N}\setminus (k_n)$ such that $y=\sum_{n\in F}x_n\in (x-b,x-a)$. Since $x-y\in (a,b)$ one can find $G\subset (k_n)$ such that $x-y=\sum_{n\in G}x_n$. Thus $x=\sum_{n\in F\cup G}x_n$. Note that $F\cup G\in I$, so $x\in  A_{I}(x_n)$. Hence $A_{I}(x_n)=\mathbb{R}$. On the other hand for each $A\in I$ we have $\sum_{n\in A}\vert x_n\vert<\infty$, so the second condition in \ref{filipowszuca} is not satisfied.
\end{proof}
\end{corollary}

Here we complete the characterization of Filip\'ow and Szuca. We show that $SR_I(x_n)$ can be a singleton, the whole line or halfline, while $SR(x_n)$ can be a singleton or $\R$ only.

\begin{proposition}\label{wytracaniegranicy}
Let $\sum_{n=1}^{\infty}x_n$ be a divergent series of positive terms such that $\lim_{n\to\infty}x_n=0$. Then for any $x\geq 0$ there exists $\sigma\in S_{\infty}$ such that $\sum_{n=1}^{\infty}(x_n-x_{\sigma(n)})=x$. 
\begin{proof}
For $x=0$ we simply take $\sigma=id$. Fix $x>0$. One can find $k_1\geq 2$ such that $\sum_{n=1}^{k_1-1}x_n\leq  x+x_1$ and $\sum_{n=1}^{k_1}x_n> x+x_1$. Let $m$ be such that for every $p>m$ we have $x_p<\frac{x_{k_1}}{k_1-1}$. Define $\sigma(n)=x_{m+n}$ for $n\in\{1,\ldots,k_1-1\}$, $\sigma(k_1)=1$. Hence 
$$x+x_{k_1}=x+x_1+x_{k_1}-x_1\geq\sum_{n=1}^{k_1}x_n - x_1>\sum_{n=1}^{k_1}(x_n-x_{\sigma(n)})\geq x+x_1-x_1-x_{k_1}=x-x_{k_1}.$$
Assume that we have defined $\sigma(n)$ for $n\in\{1,\ldots,r\}$ for some $r\geq k_1$. Consider two cases:
\begin{enumerate}
 \item  if  $\sum_{n=1}^{r}(x_n-x_{\sigma(n)})\geq x$, then let $l\in\mathbb{N}$ be such that $x_l=\max\{x_k: k\notin\sigma(\{1,\ldots,r\})\}$; we put $\sigma(r+1)=l$.
\item if $\sum_{n=1}^{r}(x_n-x_{\sigma(n)})<x$, then denote $\alpha=x-\sum_{n=1}^{r}(x_n-x_{\sigma(n)})>0$. One can find $k_2\geq 1$ such that $\sum_{n=r+1}^{r+k_2-1}x_n\leq\alpha$ and $\sum_{n=r+1}^{r+k_2}x_n>\alpha$. Put $\delta=\sum_{n=r+1}^{r+k_2}x_n-\alpha>0$. We put $\sigma(r+1)= \min\{n: n\notin\sigma(\{1,\ldots,r\}), x_n<\min\{\frac{\delta}{k_2},x_{r+1},x_{r+2},\ldots,x_{r+k_2}\}\}$.
\end{enumerate}
We continue the construction by induction. Note that each of the conditions (1) and (2) will appear infinitely many times during this construction. Indeed, note that $\vert\sum_{n=1}^{p}(x_n-x_{\sigma(n)})-x\vert\geq \vert\sum_{n=1}^{p+1}(x_n-x_{\sigma(n)})-x\vert$ if in steps $p$ and $p+1$ the same condition ((1) or (2)) is fulfilled. On the other hand if between steps $p$ and $p+1$ the condition changes (from (1) to (2) or vice-versa), then $\vert\sum_{n=1}^{p+1}(x_n-x_{\sigma(n)})-x\vert\leq\vert x_{p+1} - x_{\sigma(p+1)}\vert$. Since $\lim_{n\to\infty} x_n = \lim_{n\to\infty} x_{\sigma(n)} = 0$, then $\sum_{n=1}^{\infty}(x_n-x_{\sigma(n)})=x$.

\end{proof}
\end{proposition}
\begin{remark}
Proposition \ref{wytracaniegranicy} is a generalization and strengthening of the Riemann Theorem. 
Let $\sum_{n=1}^{\infty}x_n$ be a conditionally convergent series with a limit $y$. 
To obtain $\sum_{n=1}^{\infty}x_{\sigma(n)}=x$ for a given $x\in\mathbb{R}$ we define  $\sigma$ as follows: if $x<y$, then by Proposition \ref{wytracaniegranicy} there exists $\sigma$ with $supp(\sigma)\subset \{n :x_n>0\}$ such that 
$$y-x=\sum_{n=1}^{\infty}(x_n-x_{\sigma(n)})=\sum_{n=1}^{\infty}x_n-\sum_{n=1}^{\infty}x_{\sigma(n)}=y-\sum_{n=1}^{\infty}x_{\sigma(n)}.$$ 
Otherwise we use  Proposition \ref{wytracaniegranicy} to find an appropriate $\sigma$ with $supp(\sigma)\subset \{n :x_n<0\}$.
\end{remark}

\begin{remark}\label{nierosnace}
Note that if $\sum_{n=1}^{\infty}x_n$ satisfies the assumption of Proposition \ref{wytracaniegranicy}  and the terms tend monotonously to $0$, then for any $\sigma\in S_{\infty}$ we have $\sum_{n=1}^{\infty}(x_n-x_{\sigma(n)})\geq 0$. Indeed, since the terms are non-increasing for every $k\in\mathbb{N}$, we have $\sup\{\sum_{n=1}^{k}x_{\sigma(n)} : \sigma\in S_{\infty}\}=\sum_{n=1}^{k}x_n$. Thus, for every $k\in\mathbb{N}$ and $\sigma\in S_{\infty}$ we get $\sum_{n=1}^{k}x_n-\sum_{n=1}^{k}x_{\sigma(n)}\geq 0$. Hence by  Proposition \ref{wytracaniegranicy} we get the equality $\{\sum_{n=1}^{\infty}(x_n-x_{\sigma(n)}):\sigma\in S_{\infty}\}=[0,\infty)$.
\end{remark}
Note that the additional assumption of monotonous convergence to $0$ for the terms of series is not just the case of taking subseries. Indeed the following example shows that there exists a divergent series $\sum_{n=1}^{\infty}x_n$ of positive terms with $\lim_{n\to\infty}x_n=0$ such that each of its subseries with non-increasing terms is convergent.
\begin{example}
Let $(x_n)=(\frac{1}{2}, \frac{1}{4},\frac{1}{3},\frac{1}{8},\frac{1}{7},\frac{1}{6},\frac{1}{5},\frac{1}{16},\frac{1}{15},\frac{1}{14},\frac{1}{13},\frac{1}{12},\frac{1}{11},\frac{1}{10},\frac{1}{9},\frac{1}{32},\ldots)$. 
\\Hence $\sum_{n=1}^{\infty}x_n=\sum_{n=1}^{\infty}\frac{1}{n}=\infty$. Moreover for any subseries $\sum_{n=1}^{\infty}x_{k_n}$ of non-increasing terms we have $\sum_{n=1}^{\infty}x_{k_n}\leq\frac{1}{2}+\frac{1}{3}+\frac{1}{5}+\frac{1}{9}+\frac{1}{17}+\ldots=\sum_{n=0}^{\infty}\frac{1}{2^n+1}<\infty$.
\end{example}
\begin{corollary}
Fix $a\leq 0$ . Then there exists a divergent series of positive terms $\sum_{n=1}^{\infty}x_n$ such that  $\{\sum_{n=1}^{\infty}(x_n-x_{\sigma(n)}):\sigma\in S_{\infty}\}=[a,\infty)$.
\begin{proof}
We need to construct a series $\sum_{n=1}^{\infty}x_n$, which satisfies two conditions: 
\begin{itemize}
\item[(i)] for every $\sigma\in S_{\infty}$  we have  $\sum_{n=1}^{\infty}(x_n-x_{\sigma(n)})\geq a$.
\item[(ii)] for any $x\geq a$ there exists $\sigma\in S_{\infty}$ such that $\sum_{n=1}^{\infty}(x_n-x_{\sigma(n)})=x$; 
\end{itemize}
Let $\sum_{n=1}^{\infty}y_n$ be a divergent series of positive, non-increasing terms. By Proposition \ref{wytracaniegranicy} let $\tau\in S_{\infty}$ be such that $\sum_{n=1}^{\infty}(y_n-y_{\tau(n)})=-a$. Let us consider $\sum_{n=1}^{\infty}x_n$ with $x_n=y_{\tau(n)}$ for each $n\in\mathbb{N}$. By Remark \ref{nierosnace} we know that for any $\sigma\in S_{\infty}$ we have $\sum_{n=1}^{\infty}(y_n-x_{\sigma(n)})\geq 0$, so  $\sum_{n=1}^{\infty}(x_n-x_{\sigma(n)})=\sum_{n=1}^{\infty}(x_n-y_n+y_n-x_{\sigma(n)})\geq a$. Hence we obtain (i).
\\Fix $x=a+b$ for some $b\geq 0$. By Proposition \ref{wytracaniegranicy}, we can find $\pi\in  S_{\infty}$ such that  $\sum_{n=1}^{\infty}(y_n-y_{\pi(n)})=b$. Let $\sigma=\tau^{-1}(\pi)$. Thus $\sum_{n=1}^{\infty}(x_n-x_{\sigma(n)})=\sum_{n=1}^{\infty}(x_n-y_n+y_n-x_{\tau^{-1}(\pi)(n)})=\sum_{n=1}^{\infty}(x_n-y_n+y_n-y_{\pi(n)})=a+b$, which gives us (ii).
\end{proof}
\end{corollary}
\begin{proposition}\label{polprosta}
Let $\sum_{n=1}^{\infty}x_n$ be a divergent series of positive terms such that $\lim_{n\to\infty}x_n=0$ and let $S=\{\sum_{n=1}^{\infty}(x_n-x_{\sigma(n)}):\sigma\in S_{\infty}\}$. If $b\in S$, then $[b,\infty)\subset S$.
\begin{proof}
Let $b\in S$, $b=\sum_{n=1}^{\infty}(x_n-x_{\sigma(n)})$.  Fix $y\in[b,\infty)$ and denote $x=y-b\geq 0$. From Proposition \ref{wytracaniegranicy} applied to  $\sum_{n=1}^{\infty}x_{\sigma(n)}$ one can find $\tau\in S_{\infty}$ such that $\sum_{n=1}^{\infty}(x_{\sigma(n)}-x_{\tau(\sigma(n))})=x$. Hence 
 $$\sum_{n=1}^{\infty}(x_n-x_{\tau(\sigma(n))})=\sum_{n=1}^{\infty}(x_n-x_{\sigma(n)})+\sum_{n=1}^{\infty}(x_{\sigma(n)}-x_{\tau(\sigma(n))})=b+x=y$$
\end{proof}
\end{proposition}
\begin{proposition} \label{prosta}
There exists a divergent series of positive terms $\sum_{n=1}^{\infty}y_n$ such that $\lim_{n\to\infty}y_n=0$, for which $S=\{\sum_{n=1}^{\infty}(y_n-y_{\tau(n)}):\tau\in S_{\infty}\}=\mathbb{R}$.
\begin{proof}
Let $\sum_{n=1}^{\infty}x_n$ be a divergent series of positive terms such that $\lim_{n\to\infty}x_n=0$.
Let $A_1\subset\mathbb{N}$ be such that $\sum_{n\in A_1}x_n=\infty=\sum_{n\in \mathbb{N}\setminus A_1}x_n$. By Proposition \ref{wytracaniegranicy} let $\sigma_1\in S_{\infty}$ with $supp(\sigma_1)=A_1$ be such that $\sum_{n=1}^{\infty}(x_n-x_{\sigma_1(n)})=1$. We construct inductively a sequence $(A_k)$ of subsets of $\mathbb{N}$  such that $A_{k+1}\subset\mathbb{N}\setminus A_k$ and  $\sum_{n\in A_k}x_n=\infty=\sum_{n\in\mathbb{N}\setminus\cup_{p=1}^{k}A_p}$
for each $k\in\mathbb{N}$ and a sequence $(\sigma_k)\subset S_{\infty}$ with $supp(\sigma_k)=A_k$ such that $\sum_{n=1}^{\infty}(x_n-x_{\sigma_k(n)})=\frac{1}{k}$ for each $k\in\mathbb{N}$. Let $\sigma\in S_{\infty}$ be given by $\sigma(n)=\sigma_k(n)$ iff $n\in A_k$ and $\sigma(n)=n$ for $n\in\mathbb{N}\setminus\bigcup_{k=1}^{\infty}A_k$. Define $y_n=x_{\sigma(n)}$.
Fix $y<0$ such that $y=-\sum_{k\in A}\frac{1}{k}$ for a finite set of indices $A$. Let us define $\tau\in S_{\infty}$ by the formula $\tau(n)=\sigma^{-1}(n)$ for $n\in \bigcup_{k\in A}A_k$ and $\tau(n)=n$  for $n\in \mathbb{N}\setminus\bigcup_{k\in A}A_k$. Hence 
$$\sum_{n=1}^{\infty}(y_n-y_{\tau(n)})=\sum_{n\in \bigcup_{k\in A}A_k}(y_n-y_{\tau(n)})=\sum_{k\in A}\sum_{n\in A_k}(y_n-y_{\tau(n)})=\sum_{k\in A}\sum_{n\in A_k}(x_{\sigma_k(n)}-x_{n})=-\sum_{k\in A}\frac{1}{k}=y$$
Hence $S\supset \{-\sum_{n\in A}\frac{1}{n} : \vert A\vert<\infty\}$. 
Fix $z\in\mathbb{R}$. Then there exists $r\in\mathbb{N}$ such that $z\geq-\sum_{n=1}^{r}\frac{1}{n}$. Since $-\sum_{n=1}^{r}\frac{1}{n}\in S$, by Proposition \ref{polprosta}, we obtain that $z\in S$. Hence $S=\mathbb{R}$.
\end{proof}
\end{proposition}
\begin{corollary}
Let $\sum_{n=1}^{\infty}x_n$ be a divergent series of positive terms such that $\lim_{n\to\infty}x_n=0$. Then the set $S=\{\sum_{n=1}^{\infty}(x_n-x_{\sigma(n)}):\sigma\in S_{\infty}\}$ is either $\mathbb{R}$ or a halfline, bounded from below. 
\begin{proof}
Combine Propositions \ref{polprosta} and \ref{prosta}.
\end{proof}
\end{corollary}
\begin{theorem}\label{punktyprostepolproste}
Let $\sum_{n=1}^{\infty}x_n$ be a conditionally convergent series of reals and $I$ be an ideal. Then:
\begin{enumerate}
\item $SR_{I}(x_n)=\{\sum_{n=1}^{\infty}x_n\}$ if and only if for every $A\in I$ we have $\sum_{n\in A}\vert x_n\vert<\infty$;
\item $SR_{I}(x_n)=\mathbb{R}$  if there exists $A\in I$ such that $\sum_{n\in A}x_n^{+}=\sum_{n\in A}x_n^{-}=\infty$, where $x^{+}=\max\{x,0\}$,  $x^{-}=\max\{-x,0\}$ for every $x\in\mathbb{R}$;
\item $SR_{I}(x_n)\supset(-\infty,\sum_{n=1}^{\infty}x_n]$ if there exists $A\in I$ such that $\sum_{n\in A}x_n^{+}=\infty$ and for every $A\in I$ such that $\sum_{n\in A}x_n^{+}=\infty$ we have $\sum_{n\in A}x_n^{-}<\infty$;
\item  $SR_{I}(x_n)\supset[\sum_{n=1}^{\infty}x_n,\infty)$ if there exists $B\in I$ such that $\sum_{n\in B}x_n^{-}=\infty$ and for every $B\in I$ such that $\sum_{n\in B}x_n^{-}=\infty$ we have $\sum_{n\in B}x_n^{+}<\infty$.
\end{enumerate}
Note that for each conditionally convergent series exactly one of the four conditions imposed on the series above holds. Indeed if $(1)$ holds then neither $(2)$ nor $(3)$ nor $(4)$ hold. If $(3)$ or $(4)$ hold then $(2)$ does not hold and vice-versa. Moreover if we suppose that $(3)$ and $(4)$ are both satisfied and $A\in I$ and $B\in I$ are such that $\sum_{n\in A}x_n^{+}=\infty$ and $\sum_{n\in B}x_n^{-}=\infty$, then $A\cup B\in I$ and $\sum_{n\in A\cup B}x_n^{+}=\sum_{n\in A\cup B}x_n^{-}=\infty$, so $(2)$ holds, which gives us a contradiction.
\begin{proof}
Proofs of $(1)$ and $(2)$ are obvious. 
\\Let us assume that a conditionally convergent series $\sum_{n=1}^{\infty}x_n$ satisfies $(3)$. Let $A\in I$ be such that $\sum_{n\in A}x_n^{+}=\infty$. By simply taking the subset of $A$ we may assume that the series  $\sum_{n\in A}x_n^{+}$ has positive terms, that is $x_n^{+}=x_n$. Fix $y\in (-\infty,\sum_{n=1}^{\infty}x_n]$. We use Proposition \ref{wytracaniegranicy} for $x=\sum_{n=1}^{\infty}x_n-y\geq 0$. Let $\sigma$ be such that $\sum_{n\in A}(x_n-x_{\sigma(n)})=x$. Define $\tau(n)=\sigma(n)$ for $n\in A$ and  $\tau(n)=n$ for $n\in\mathbb{N}\setminus A$. Thus $\sum_{n=1}^{\infty}x_{\tau(n)}=\sum_{n=1}^{\infty}x_n+y-\sum_{n=1}^{\infty}x_n=y$.
\\The proof of $(4)$ is very simillar to $(3)$.
\end{proof}
\end{theorem}
\begin{remark}
Note that the implication (2) in Theorem \ref{punktyprostepolproste} cannot be reversed. Indeed, by Proposition \ref{prosta} we get that the equality $SR_{I}(x_n)=\mathbb{R}$ can hold when the assumptions of (3) or (4) are satisfied.
\end{remark}


 

\section{Complexity of ideally supported achievement sets}

Let us start from presenting the following examples.
\begin{example}\label{fsigma}
Let $x_n=\frac{2}{3^n}$ for $n\in\mathbb{N}$ and $I=Fin$. Note that $Fin$ is an $F_{\sigma}$-set, which is not a  $G_{\delta}$-set. Since $f$ is a homeomorphism, we obtain that also $A_{Fin}(x_n)$  is an $F_{\sigma}$-set, which is not a  $G_{\delta}$-set. Moreover for any $J\supset Fin$ we know that $J$ is not a $G_{\delta}$, so $A_{J}(x_n)$ is not a  $G_{\delta}$.
\end{example}
\begin{example}\label{fsigmadelta}
Let $x_n=\frac{2}{3^n}$  for $n\in\mathbb{N}$ and $I=I_d$. In \cite{HL} the authors proved that  $I_d$ is an $F_{\sigma\delta}$-set, which is not a $G_{\delta\sigma}$-set. Hence $A_{I_d}(x_n)$  is an $F_{\sigma\delta}$-set, which is not a  $G_{\delta\sigma}$-set.
\end{example}

\begin{theorem}
Let $I_d$ be a ideal of statistical density zero. Then $A_{I_d}(\frac{1}{2^n})$ is a null subset of $[0,1]=A(\frac{1}{2^n})$.  
\begin{proof}
 Indeed by the Borel's Theorem on Normal Numbers the set $F=\{\sum_{n\in B} \frac{1}{2^n} :  \lim_{n\to\infty} \frac{B\cap\{1,\ldots,n\}}{n}=\frac{1}{2}\}$ has Lebesgue measure $1$, for the proof see \cite{Walters}. Suppose that there exists $x\in F\cap A_{I_d}(x_n)$. 
Then there exists $A\in I_d$ and $B$ with $\lim_{n\to\infty} \frac{B\cap\{1,\ldots,n\}}{n}=\frac{1}{2}$ such that $x=\sum_{n\in A} \frac{1}{2^n}=\sum_{n\in B} \frac{1}{2^n}$. Clearly $B\notin I_d$, so $A\neq B$. Observe that almost every point $x\in [0,1]$ has a unique representation by the set $E$ of those indices $n$, such that $x=\sum_{n\in E} \frac{1}{2^n}$. Hence $x=\frac{m}{2^k}$ for some $k\in\mathbb{N}$ and $m\in\{1,\ldots,2^k\}$. It implies that $A\subset\{1,\ldots,k+1\}$, which gives us the inclusion $B\supset\{k+2,k+3,\ldots\}$. Thus $\lim_{n\to\infty} \frac{B\cap\{1,\ldots,n\}}{n}=1$, which yields a contradiction. Thus $F\cap A_{I_d}(x_n)=\emptyset$. Since $\lambda(F)=1$, we get that  $A_{I_d}(x_n)$ is null. 
\end{proof}
\end{theorem}


Examples \ref{fsigma} and \ref{fsigmadelta} recall that if a series' associated function $f$ is a homeomorphism, then ideal achievement sets are usually of a high Borel class.
Now we will show the opposite of that fact, namely if $f$ is not an injection, then we can have more regular ideal achievement sets. In particular for $I\neq Fin$ we can obtain that $f(I)$  is a compact set up to some finite set, see \ref{brakujacysingleton}, and if $I$ is not maximal then $f(I)$ can even be an open set, see \ref{idealmaksymalnyizbiorotwarty}. 

We have the following inclusions $A_{Fin}(x_n)\subset A_{I}(x_n)\subset A(x_n)$. Now we will study if these inclusions have to be strict or not. The simple observation shows that if $(x_n)\in c_{00}$ then  $A_{Fin}(x_n)= A(x_n)$. 
Moreover if infinitely many of terms of our series are equal to zero and $\{n : x_n\neq 0\}\in I$, then $A_{I}(x_n)=A(x_n)$. 
\begin{proposition}\label{finprzeliczalny}
Let $I\neq Fin$ be an ideal and $(x_n)\in  \ell_1^{*}$ . Then $A_{Fin}(x_n)$ is a strict subset of $A_{I}(x_n)$.
\begin{proof}
Note that  $A_{Fin}(x_n)=\{\sum_{n\in A} x_n : A\in Fin\}=\{\sum_{n=1}^{k} \varepsilon_n x_n : (\varepsilon_n)_{n=1}^{k}\in\{0,1\}^k, k\in\mathbb{N}\}$ and hence it is countable. Since $(x_n)\in  \ell_1^{*}$ then one can find a subsequence $(m_n)_{n=1}^{\infty}\in I$ such that $\vert x_{m_{n+1}}\vert <\frac{\vert x_{m_{n}}\vert}{2}$ for each $n\in\mathbb{N}$. Note that $\{0,1\}^{\mathbb{N}}\ni (\delta_n)\rightarrow \sum_{n=1}^{\infty} \delta_n x_{m_n}$ is one-to-one, so $A_I(x_{n})$ is uncountable. Hence $A_{Fin}(x_n)\neq A_{I}(x_n)$.
\end{proof}
\end{proposition}

\begin{proposition}
For every $(x_n)\in  \ell_1^{*}$, there exists an ideal $I\neq Fin$ such that $A_I(x_n)$ is meager and null.
\end{proposition}
\begin{proof}
 Let $(x_n)\in  \ell_1^{*}$. Then for $I=B+Fin$, where $B=\{m_n: n\in\mathbb{N}\}$ is defined as in \ref{finprzeliczalny}, we get that $A_{I}(x_n)$ is a subset of an algebraic sum of $A_{Fin}(x_n)$ and a set which is homeomorphic to the Cantor set. Since $A_{Fin}(x_n)$ is countable we get that $A_{I}(x_n)$ is meager.
 Moreover if we take $(m_n)_{n=1}^{\infty}\in I$ with $\vert x_{m_{n+1}}\vert <\frac{\vert x_{m_{n}}\vert}{3}$ for each $n\in\mathbb{N}$, then 
 by the formula given in \cite{BFPW} we have $\mu(A(x_{m_n}))=\lim_{k\to\infty} 2^{k}r_k$. Note that $r_k= \sum_{n=k+1}^{\infty}  x_{m_n}\leq \sum_{n=k+1}^{\infty}  \vert x_{m_n}\vert \leq  \vert x_{m_{k+1}}\vert\sum_{n=0}^{\infty}  3^{-n}\leq \frac{3}{2}\vert x_{m_1}\vert 3^{-k}$. Hence $\mu(A(x_{m_n}))=0$. Thus $\mu(A_{I}(x_n))=0$.
\end{proof}
\begin{proposition}\label{idealowywlasciwy}
Let $I$ be an ideal and $(x_n)\in  \ell_1^{*}$. Then $A_{I}(x_n)$ is a strict subset of $A(x_n)$.
\begin{proof}
Let $A=\{n : x_n>0\}$, then $\mathbb{N}\setminus A=\{n : x_n<0\}$. Put $x=\sum_{n\in A}x_n,y=\sum_{n\in \mathbb{N}\setminus A}x_n$. Then $x,y$ are obtained in the unique way presented above and $x,y\in A(x_n)$. Suppose that $x,y\in A_{I}(x_n)$. Thus $A\in I$ and $\mathbb{N}\setminus A\in I$, which gives us contradiction. 
\end{proof}
\end{proposition}
Note that for a convergent series $\sum_{n=1}^{\infty}x_n$ the set  $A_{Fin}(x_n)$ is a dense, countable subset of $A(x_n)$. Hence $A_{I}(x_n)$ is also a dense subset of $A(x_n)$ for each $I\supset Fin$. 
\begin{theorem}\label{brakujacysingleton}
Let $I$ be an ideal which is not equal to $Fin$. Then there exists a sequence $(x_n)\in  \ell_1^{*}$ such that $A(x_n)\setminus A_I(x_n)$ is a singleton.
\begin{proof}
Without losing generality assume that $x>0$. Let $A=\{a_1<a_2<\ldots\}\in I$ and $a_{0}=0$, $a_{i+1}>a_{i}+1$ for $i\in\mathbb{N}_{0}$. Define $x_{a_{2n-1}}=\frac{x}{2^{n+1}}$ and $x_{a_{2n}}=-\frac{x}{2^{n+1}}$ for $n\in\mathbb{N}$. Moreover let $x_n=\frac{x}{2^{i+2}(a_{i+1}-a_i-1)}$ if $n\in\{a_{i}+1,\ldots,a_{i+1}-1\}$ for every $i\in\mathbb{N}_0$. Note that $(x_n)\in  \ell_1^{*}$ and it satisfies the following equalities $\sum_{k=a_{i}+1}^{a_{i+1}-1}x_k=\frac{x}{2^{i+2}}$ for every $i\in\mathbb{N}_{0}$. By the construction of $(x_n)$ we get $A_{I}(x_n)\supset [-\frac{x}{2},\frac{x}{2}]$. Moreover $A(x_n)=[-\frac{x}{2},x]$. Fix $z\in (\frac{x}{2},x)$. Since $\sum_{n\in\mathbb{N}\setminus A} x_n=\frac{x}{2}$ one can find a finite set $D\subset \mathbb{N}\setminus A$ such that $\frac{x}{2}>\sum_{n\in D} x_n>z-\frac{x}{2}>0$. Let $E\subset A$ be such that $\sum_{n\in E}x_n=z-\sum_{n\in D}x_n$. Put $F=D\cup E$. Since $D\in I$ and $E\in I$ we get that $F\in I$. Note that $z=\sum_{n\in F}x_n$, so $z\in A_{I}(x_n)$. Moreover $x=\sum_{n\in \mathbb{N}\setminus(a_{2n})_{n=1}^{\infty}}x_n$ is obtained in only that way. Since $(a_{2n})_{n=1}^{\infty}\subset (a_{n})_{n=1}^{\infty}$ and $(a_{n})_{n=1}^{\infty}\in I$ we get that $(a_{2n})_{n=1}^{\infty}\in I$. Hence $\mathbb{N}\setminus(a_{2n})_{n=1}^{\infty}\notin I$, so $x\in A(x_n)\setminus A_I(x_n)$.
\end{proof}
\end{theorem}
\begin{remark}
Note that by Proposition \ref{idealowywlasciwy}, the point $x$ from Theorem \ref{brakujacysingleton}  has to be either $\max A(x_n)$ or $\min A(x_n)$. The proof of Theorem \ref{brakujacysingleton} shows that it is possible to obtain $A(x_n)\setminus A_I(x_n)=\max A(x_n)$.
\end{remark}
\begin{theorem}\label{idealmaksymalnyizbiorotwarty}
Let $I$ be an ideal which covers and is not equal to $Fin$.  The following assertions are equivalent:
\begin{enumerate}
\item there exists a sequence $(x_n)\in \ell_1^{*}$ such that $A_I(x_n)$ is open 
\item $I$ is not maximal 
\end{enumerate}
\begin{proof}
$\Rightarrow$. Suppose that $(x_n)\in  \ell_1^{*}$ is such that $A_I(x_n)$ is open and $I$ is maximal. Let $A=\{n : x_n>0\}$. If $A=\emptyset$ or $A=\mathbb{N}$ then $0\in A_I(x_n)$  and $A_I(x_n)\cap (0,\infty)=\emptyset$ or $A_I(x_n)\cap (-\infty,0)=\emptyset$ respectively. Hence  $A_I(x_n)$ is not open. Assume that $\emptyset\neq A\neq\mathbb{N}$. Then $A\in I$ or $\mathbb{N}\setminus A\in I$. Without loss of generality assume that $A\in I$ and fix $x=\sum_{n\in A} x_n$. Then $x\in A_I(x_n)$ and $A_I(x_n)\cap (x,\infty)=\emptyset$, so $A_I(x_n)$ is not open. If $\mathbb{N}\setminus A\in I$, then by a simillar reasoning we get that $A_I(x_n)$ is not open, which gives us contradiction.
\\$\Leftarrow$. Assume that $I$ is not maximal. Let $A\subset\mathbb{N}$ be such that $A\notin I$ and $B=\mathbb{N}\setminus A\notin I$. Let $C\in I\setminus Fin$. Then $A\cap C$ is infinite or $B\cap C$ is infinite. Without loss of generality assume that $D=A\cap C$ is infinite. Since $D\subset C$ we have $D\in I$ and $E=A\setminus C\notin I$. Denote $B=\{b_1<b_2<\ldots\}$, $D=\{d_1<d_2<\ldots\}$, $E=\{e_1<e_2<\ldots\}$. Define $x_{{d_n}}=\frac{1}{2^{n}}$, $x_{{b_n}}=-\frac{1}{2^{n}}$, $x_{{e_n}}=\frac{1}{2^{n}}$ for every $n\in\mathbb{N}$. We have $A(x_n)=[-1,2]$. Since $D\in I$ we get $A_{I}(x_{n})\supset\{\sum_{n\in F} x_n : F\subset D\}=[0,1]$. Fix $x\in (1,2)$. One can find finite subset $G$ of $E$ such that $1>\sum_{n\in G} x_n>x-1$. There exists $H\subset D$ such that $\sum_{n\in G} x_n+\sum_{n\in H} x_n=\sum_{n\in G\cup H} x_n=x$. Since $H\in I$ and $G\in I$ we have $G\cup H\in I$, so $x\in A_{I}(x_n)$. Hence $A_{I}(x_n)\supset (1,2)$. In the simillar way we prove that $A_{I}(x_n)\supset (-1,0)$. We get $A_{I}(x_n)\supset (-1,2)$. Observe that $\sum_{n\in W} x_n=2$ if and only if $W=D\cup E=A\notin I$ and $\sum_{n\in U} x_n=-1$ if and only if $U=B\notin I$. Hence $2\notin A_{I}(x_n)$ and $-1\notin A_{I}(x_n)$, so $A_{I}(x_n)=(-1,2)$.
\end{proof}
\end{theorem}
A simple modification of the series defined in  \ref{idealmaksymalnyizbiorotwarty} shows that if $A_{I}(x_n)$ is an open subset of $A(x_n)$, then $A_{I}(x_n)$ does not have to be the interior of $A(x_n)$.
\begin{example}
Let $I$ be an ideal, which is not maximal and $(x_n)$ be the sequence defined in the second part of  \ref{idealmaksymalnyizbiorotwarty}. Define $y_{n+1}=x_n$ for every $n\in\mathbb{N}$ and $y_{1}=3$. Then $A(y_n)=[-1,5]$ and $A_{I}(y_n)=(-1,2)\cup (2,5)$. 
\end{example}
\begin{remark}\label{bezwlasnoscibairea}
Let $x_n=\frac{1}{2^n}$ for each $n\in\mathbb{N}$ and $I$ be maximal. Then $A_I(x_n)$ is a non-measurable set, which does not satisfy the Baire property. In particular, it means that $A_I(x_n)$ is not a Borel set. 
\end{remark}
The next theorem shows how different the properties of continuous functions are from those of homeomorphisms.
\begin{example}
One can construct a continuous function with a domain, which is non-measurable, without Baire's property and an open image.
\end{example}
\begin{proof}
Let $I_1$ and $I_2$ be maximal ideals on $2\mathbb{N}-1$ and $2\mathbb{N}$ respectively. Let us define $I=\{A\subset\mathbb{N} : A\cap 2\mathbb{N}-1\in I_1, A\cap 2\mathbb{N}\in I_2\}$. Then : $2\mathbb{N}-1\notin I$ and $2\mathbb{N}\notin I$, so $I$ is not maximal. Since $\{0,1\}^\mathbb{N}=\{0,1\}^{2\mathbb{N}-1}\times\{0,1\}^{2\mathbb{N}}$ we may view $I$ as a product $I_1\times I_2$. Thus $I$ neither has Baire property (by the Kuratowski--Ulam Theorem \cite[8.41]{Kechris}) nor is measurable (by the Fubini Theorem). 
 We construct $(x_n)$ in the same way as in \ref{idealmaksymalnyizbiorotwarty}. Thus $f: I\rightarrow\mathbb{R}$ defined as $f: \{0,1\}^{\mathbb{N}}\supset \chi_A\rightarrow \sum_{n\in A}x_n$ completes the proof.
\end{proof}
We introduce the following definition 
\begin{definition}\label{supset}
Let $I$ be an ideal which is not dense and let $A$ be such that $\vert A\vert=\infty$ and for every $B\subset A$,  $B\in I$ we have $B\in Fin$. We say that $I$ has \emph{supset property} if there exists $C\supset A$ such that for every $D\subset C$,  $D\in I$ we have $D\in Fin$ and $\mathbb{N}\setminus C\in I$.
\end{definition}
Note that if $I\neq Fin$, $I\supset Fin$ then $\mathbb{N}\setminus C\notin Fin$.
\begin{example}
Let  $\vert E\vert=\infty$, $\vert\mathbb{N}\setminus E\vert=\infty$ and $I=E+Fin$. We will use the notation of \ref{supset}.
\\Note that $A$ satisfies the assumptions of  \ref{supset} if and only if $A\cap E\in Fin$. Hence $I$ is not dense. Put $C=A\cap E\cup(\mathbb{N}\setminus E)$. Then for every $D\subset C$, $D\in I$ we have $D\cap E\in Fin$ and $D\cap(\mathbb{N}\setminus E)\in Fin$, so $D\in Fin$. Note that $\mathbb{N}\setminus C\subset E$, so $\mathbb{N}\setminus C\in I$. Hence $I$ has the supset property.
\end{example}
\begin{theorem}
Let $I$ be an ideal which has the supset property. Assume that $C$ satisfies the assumptions of \ref{supset}. Then there exists $(x_n)\in \ell_1^{*}$ such that $A(x_n)$ is an interval and $A_I(x_n)$ is meager and null.
\begin{proof}
Denote $C=\{c_1<c_2<\ldots\}$, $\mathbb{N}\setminus C=\{d_1<d_2<\ldots\}$. Define $x_{c_i}=\frac{1}{2^i}$, $x_{d_i}=\frac{2}{3^i}$ for $i\in\mathbb{N}$. 
Note that $A(x_{c_i})=[0,1]$ and  $A(x_{d_i})$ is the ternary Cantor set. Hence $A(x_n)=[0,2]$. Moreover $A_I(x_n)=\{\sum_{n\in F} x_n : F\in I \}=\{\sum_{n\in G\cup H} x_n : G\in Fin\cap C, H\subset\mathbb{N}\setminus C\}=A+B$, where $A=\{\sum _{n=1}^{k}\frac{\varepsilon_{n}}{2^{n}}: (\varepsilon_{n})\in\{0,1\}^{k}, k\in\mathbb{N}\}$ is the set of all dyadic numbers from interval $[0,1)$ and $B$ is the ternary Cantor set. Hence $A_I(x_n)$ is a null, meager and dense subset of $[0,2]$ as a countable union of nowhere dense, null sets.
\end{proof}
\end{theorem}
\begin{theorem}
For every ideal $I$, which is meager and null, the achievement set $A_I(\frac{1}{2^n})$ is a meager and null subset of $[0,1]=A(\frac{1}{2^n})$. In particular the above holds for Borel ideals. 
\begin{proof}
As we mentioned in the Background Section, the function $f(\chi_A)=\sum_{n\in A}\frac{1}{2^n}$, $f:\{0,1\}^\N\to[0,1]$ preserves meager and null sets. Since a Borel ideal $I$ is meager and null, so is $A_I(\frac{1}{2^n})=f(I)$. 
\end{proof}
\end{theorem}
\begin{theorem}
Assume that $I$ is maximal. If $A(x_n)$ is nonmeager, then $A_{I}(x_n)$ is nonmeager. If $\lambda^{\ast}(A(x_n))>0$ and $A_{I}(x_n)$ is measurable, then $\lambda^{\ast}(A_{I}(x_n))>0$, where $\lambda^{\ast}$ is the outer Lebesgue measure.
\end{theorem}
\begin{proof}
Let  $A(x_n)$ be nonmeager. Denote $B=[0,\frac{1}{2}\sum_{n=1}^{\infty}x_n)\cap A_{I}(x_n)$. If $B$ is nonmeager, then $A_{I}(x_n)\supset B$ is also nonmeager.
Suppose that $B$ is meager. Since $\frac{1}{2}\sum_{n=1}^{\infty}x_n$ is a point of reflection of $A(x_n)$, we get that $C=[0,\frac{1}{2}\sum_{n=1}^{\infty}x_n)\cap A(x_n)$ is nonmeager, so $C\setminus B$ is also nonmeager. Fix $x\in C\setminus B$. Then $x\in A(x_n)$ and $x\notin A_{I}(x_n)$. Thus $x=\sum_{n\in A}x_n$ for some $A\notin I$. Hence $\sum_{n=1}^{\infty}x_n-x =\sum_{n\in\mathbb{N}\setminus A}x_n\in A_{I}(x_n)$. We showed that $\sum_{n=1}^{\infty}x_n-(C\setminus B)\subset A_{I}(x_n)$. Note that $\sum_{n=1}^{\infty}x_n-(C\setminus B)$ is nonmeager, so $A_{I}(x_n)$ is nonmeager.
Observe that if we replace the properties of the defined sets of being  nonmeager and  meager  with being nonnull and null, respectively, then the statement also holds.
\end{proof}
\begin{corollary}
Assume that $I$ is maximal and $A=\{x\in A(x_n) : x=\sum_{n=1}^{\infty}\varepsilon_n x_n \ \text{for the unique sequence} (\varepsilon_n) \}$ is comeager in $A(x_n)$. Then $A_{I}(x_n)$ cannot be comeager in $A(x_n)$.
\end{corollary}
\begin{proof}
Suppose that $A_{I}(x_n)$ is comeager in $A(x_n)$. Then $A\cap A_{I}(x_n)$ is comeager in $A(x_n)$.
Since $B=A\cap A_{I}(x_n)\cap [0,\frac{1}{2}\sum_{n=1}^{\infty}x_n)$ is comeager in $A(x_n)\cap [0,\frac{1}{2}\sum_{n=1}^{\infty}x_n)$ we get $(\sum_{n=1}^{\infty}x_n - B)\cap A_{I}(x_n)$ is meager. Note that $A_{I}(x_n)\cap(\frac{1}{2}\sum_{n=1}^{\infty}x_n,\sum_{n=1}^{\infty}x_n)\subset ((\sum_{n=1}^{\infty}x_n - B)\cap A_{I}(x_n))\cup (\mathbb{N}\setminus A)$, so $A_{I}(x_n)$ is meager in $(\frac{1}{2}\sum_{n=1}^{\infty}x_n,\sum_{n=1}^{\infty}x_n)$, which gives a contradiction.
\end{proof}

\section{Symmetrization of ideal achievement sets}
The achievement set $A(x_n)$ is symmetric, while its ideal counterpart $A_I(x_n)$ lacks symmetry. To fix the symmetry we add to $A_I(x_n)$ its filter counterpart $A_{F_I}(x_n)$. 
 Simply observations shows that if $x=\sum_{n\in A}x_n$ for some $A\in I$ then $\sum_{n=1}^{\infty}x_n-x=\sum_{n\in\mathbb{N}\setminus A}x_n\in A_{F_I}(x_n)$ and vice versa. Hence $A_{F_I}(x_n)=\sum_{n=1}^{\infty}x_n-A_I(x_n)$.
 It is clear that $A_{I}(x_n)\cap A_{F_I}(x_n)$ and $A_{I}(x_n)\cup A_{F_I}(x_n)$  are symmetric with a point of reflection $\frac{1}{2}\sum_{n=1}^{\infty}x_n$. Moreover if $I$ is maximal, then $A_{I}(x_n)\cup A_{F_I}(x_n)=A(x_n)$.
Moreover if every point of $A(x_n)$ is obtained uniquely, then $A_{I}(x_n)\cap A_{F_I}(x_n)=\emptyset$. 
\begin{proposition}
 Let $I$ be an ideal. Then $A_{I}(x_n)\cap A_{F_I}(x_n)\subset A(x_n)\setminus\{\max A(x_n),\min A(x_n)\}$.
\end{proposition}
\begin{proof}
By Proposition \ref{idealowywlasciwy} we obtain that $A_{I}(x_n)\subset A(x_n)\setminus\{\max A(x_n)\}$ or $A_{I}(x_n)\subset A(x_n)\setminus\{\min A(x_n)\}$. Since  $A_{F_I}(x_n)=\sum_{n=1}^{\infty}x_n-A_I(x_n)$ we get  $A_{F_I}(x_n)\subset A(x_n)\setminus\{\min A(x_n)\}$ or $A_{F_I}(x_n)\subset A(x_n)\setminus\{\max A(x_n)\}$ respectively.
\end{proof}
\begin{example}
Let $(x_n)$ be the sequence defined  in the proof of Theorem \ref{idealmaksymalnyizbiorotwarty}, then $A_{I}(x_n)=[-1,2]$ and $A_{I}(x_n)=(-1,2)$. Hence  $A_{F_I}(x_n)=(-1,2)$, so $A_{I}(x_n)\cap A_{F_I}(x_n)= A(x_n)\setminus\{\max A(x_n),\min A(x_n)\}$.
\end{example}
\begin{example}
Let $(x_n)$ be the sequence defined  in the proof of Theorem \ref{brakujacysingleton}, then $A_{I}(x_n)\cap A_{F_I}(x_n)= A(x_n)\setminus\{\max A(x_n),\min A(x_n)\}$ and $A_{I}(x_n)\neq A_{F_I}(x_n)$. Note that $A_{I}(x_n)\cup A_{F_I}(x_n)=A(x_n)$ despite of $I$ does not need to be maximal. 
\end{example}
Now we consider the case when $A_{I}(x_n)\cap A_{F_I}(x_n)$ is a singleton.
\begin{proposition}
Assume that $A_{I}(x_n)\cap A_{F_I}(x_n)=\{x\}$. Then 
\begin{enumerate}
\item $x=\frac{1}{2}  \sum_{n=1}^{\infty}x_n$
\item if $x=\sum_{n\in A}x_n=\sum_{n\in\mathbb{N}\setminus B}x_n$ for $A,B\in I$ then $A=B$
\item if $x=\sum_{n\in A}x_n=\sum_{n\in B}x_n$ for $A,B\in I$ then $A=B$ 
\end{enumerate}
\end{proposition}
\begin{proof}
\begin{enumerate}
\item Since $A_{F_I}(x_n)=\sum_{n=1}^{\infty}x_n-A_I(x_n)$, it is clear that $x=\frac{1}{2}  \sum_{n=1}^{\infty}x_n$.
\item Let $x=\sum_{n\in A}x_n=\sum_{n\in\mathbb{N}\setminus B}x_n$ for some $A,B\in I$. If there exists $k\in A\cap(\mathbb{N}\setminus B)$, then $x-x_k=\sum_{n\in A\setminus\{k\}}x_n$ and $x-x_k=\sum_{n\in\mathbb{N}\setminus (B\cup\{k\}}x_n$, so $x-x_k\in A_{I}(x_n)\cap A_{F_I}(x_n)$, which gives us a contradiction. Hence $A\cap(\mathbb{N}\setminus B)=\emptyset$.  If there exists $k\notin A\cup(\mathbb{N}\setminus B)$, then $x+x_k=\sum_{n\in A\cup\{k\}}x_n$ and $x+x_k=\sum_{n\in\mathbb{N}\setminus B\cup\{k\}}x_n$, so $x+x_k\in A_{I}(x_n)\cap A_{F_I}(x_n)$, which yields a contradiction. Hence $A\cup(\mathbb{N}\setminus B)=\mathbb{N}$. We proved that $\mathbb{N}\setminus B$ is a complement of $A$, so $A=B$. 
\item Assume that $x=\sum_{n\in A}x_n=\sum_{n\in B}x_n=\sum_{n\in\mathbb{N}\setminus A}x_n=\sum_{n\in\mathbb{N}\setminus B}x_n$ and $A\neq B$. Then $\sum_{n\in A\cup B}x_n=\sum_{n\in A}x_n+\sum_{n\in A}x_n-\sum_{n\in A\cap B}x_n=\sum_{n=1}^{\infty}x_n-\sum_{n\in A\cap B}x_n$. Since $\sum_{n\in A\cap B}x_n\in A_{I}(x_n)$, we get $\sum_{n\in A\cup B}x_n\in A_{F_I}(x_n)$. Since $A\cup B\in I$, by using the same reasoning we obtain $\sum_{n\in A\cap B}x_n\in A_{F_I}(x_n)$. Hence $\sum_{n\in A\cap B}x_n\in A_{I}(x_n)\cap A_{F_I}(x_n)$ and $\sum_{n\in A\cup B}x_n\in A_{I}(x_n)\cap A_{F_I}(x_n)$. Thus $\sum_{n\in A\cup B}x_n=\sum_{n\in A\cap B}x_n=x$. Therefore $\sum_{n\in A\setminus B}x_n=\sum_{n\in B\setminus A}x_n=0$. Since $A\neq B$, we know that $A\setminus B\neq\emptyset$ or $B\setminus A\neq\emptyset$. Assume that there exists $k\in B\setminus A$. Let $C=A\cap B\cup\{k\}\in I$, then $\sum_{n\in C}x_n = x+x_k$. Let $D=(\mathbb{N}\setminus(A\cup B))\cup\{k\}\in F_I$, then  $\sum_{n\in D}x_n=x+x_k$. Hence $x+x_k\in A_{I}(x_n)\cap A_{F_I}(x_n)$, which gives a contradiction. Hence $B\setminus A=\emptyset$ and in the same way we prove  $A\setminus B=\emptyset$. Thus $A=B$.
\end{enumerate}
\end{proof}
\begin{example}
There exists a sequence $(x_n)$ such that for each ideal $I$ we have that $A_{I}(x_n)\cap A_{F_I}(x_n)$ is a singleton.
\end{example}
\begin{proof}
Define $x_1=1$, $x_{n+1}=\frac{2}{3^n}$ for $n\in\mathbb{N}$. Note that $A(x_n)=C\cup (1+C)$, where $C$ is the ternary Cantor set.  Let $I$ be an ideal. Since $1=x_1=\sum_{n=2}^{\infty}x_n$, we get $1\in A_{I}(x_n)\cap A_{F_I}(x_n)$. Fix $x\neq 1$. There exists a unique set $A\subset\mathbb{N}$ such that $x=\sum_{n\in A}x_n$. Assume that $x\in A_{I}(x_n)\cap A_{F_I}(x_n)$. Thus $A\in I$ and $A\in F_I$, which gives a contradiction. Hence $A_{I}(x_n)\cap A_{F_I}(x_n)=\{1\}$.
\end{proof}
\section{Injectivity of the associated function}
Here we consider when the equality $A_I(x_n)=A_J(x_n)$ holds for two distinct ideals $I\neq J$. Let us consider an instructive example.
\begin{example}
Let $x_{2n-1}=\frac{1}{2^n}$, $x_{2n}=\frac{1}{2^n}$ for $n\in\mathbb{N}$ and $I=2\mathbb{N}-1+Fin$, $J=2\mathbb{N}+Fin$. Hence $A_I(x_n)=[0,2)=A_J(x_n)$. Note that $I\cap J=Fin$.
\end{example}
Now let us consider two ideals $I,J$ from which one is bigger that the other, that is $I\subset J$. We ask if it is possible to obtain  $A_I(x_n)=A_J(x_n)$. Note that in Theorem \ref{brakujacysingleton} we have constructed an ideal $I$ about which we only assumed that some sequence of indices $(a_n)\in I$, that is $I=(a_n)+Fin$ and we obtained that $A_I(x_n)=A(x_n)\setminus\{x\}$ for some $x>0$. By Proposition \ref{idealowywlasciwy} we get $A_J(x_n)=A(x_n)\setminus\{x\}$ for any $J\supset I$. The idea of Theorem \ref{brakujacysingleton} was to construct for an ideal $I$ the series for which $A_I(x_n)$ is ''big''. In this chapter we reverse this dependence, that is we solve the problem when for a series we can find two distinct ideals $I,J$ such that $A_I(x_n)=A_J(x_n)$. Clearly the series cannot be quickly convergent, since then for $I\neq J$ we always have $A_I(x_n)\neq A_J(x_n)$.
\begin{theorem}\label{rowne}
Let $\sum_{n=1}^{\infty}x_n$ be an absolutely convergent series. Let us consider the following conditions:
\begin{enumerate}
\item $f$ is injective; 
\item for every ideals $I\neq J$ we have $A_{I}(x_n)\neq A_{J}(x_n)$;
\item for every ideals $I\subsetneq J$ we have $A_{I}(x_n)\subsetneq A_{J}(x_n)$.
\end{enumerate}
Then the condition $(1)$ implies $(2)$ and the condition $(2)$ implies $(3)$. 
\begin{proof}
Proofs of both implications are clear.
\end{proof}
\end{theorem}
All three conditions look quite simillar, however none of the implications in Theorem \ref{rowne} can be reversed, which is showed by the following examples.
\begin{example}
Let us consider $x_n=\frac{1}{2^n}$. It is clear that $f$ is not injective, since each dyadic number is obtained for two sets of indices. Hence the condition $(1)$ from Proposition \ref{rowne} is not satisfied. We will show that the condition $(2)$ is satisfied. Fix two ideals $I\neq J$. Let $A\in J\setminus I$ (if $J\subsetneq I$ we simply take $A\in I\setminus J$). It is clear that $A$ is infinite since it is not an element of ideal $I$ and $A$ is not cofinite since it is an element of ideal $J$. Suppose that there exists $B\in I$ such that $x=\sum_{n\in A}x_n=\sum_{n\in B}x_n$. But it is possible only when $x$ is a dyadic number, so $A$ is finite or cofinite and we get contradiction. Hence $x\in A_{J}(x_n)\setminus A_{I}(x_n)$, so $\sum_{n=1}^{\infty}x_n$ satisfies the condition $(2)$ from Proposition \ref{rowne}.
\end{example}
\begin{example}
Let $(y_n)$ satisfy the inequality $y_n>2\sum_{k=n+1}^{\infty}y_k$ for each $n\in\mathbb{N}$. We define $x_{2n-1}=x_{2n}=y_{n}$ for every $n\in\mathbb{N}$.
Then it is clear that $A_{I}(x_n)=A_{J}(x_n)$ for $I=2\mathbb{N}-1+Fin$ and  $J=2\mathbb{N}+Fin$. Hence the condition $(2)$ from Proposition \ref{rowne} is not satisfied for the series $\sum_{n=1}^{\infty}x_n$. Now let $I\subsetneq J$. Then there exists $A\in J\setminus I$. Since $\vert A\vert=\infty$, we obtain that at least one of the sets $A\cap 2\mathbb{N}-1$ or $A\cap 2\mathbb{N}$ is infinite. Moreover if both $A\cap 2\mathbb{N}-1$ and $A\cap 2\mathbb{N}$ are infinite, then at least one of them is not an element of the ideal $I$.
Assume that $E=A\cap 2\mathbb{N}-1\notin I$.
Fix $x=\sum_{n\in E}x_n\in A_{J}(x_n)$. Suppose that $A_{J}(x_n)=A_{I}(x_n)$. Hence $x\in A_{I}(x_n)$, that is  $x=\sum_{n\in F}x_n$ for some $F\in I$. 
Note that by the definition of $(x_n)$ we have $x=\sum_{n\in E+1}x_n$, then by the condition $y_n>2\sum_{k=n+1}^{\infty}y_k$, we obtain  that $F\subset E\cup (E+1)$ and for every $2n-1\in E$ either $2n-1\in F$ or $2n\in F$. Moreover $E\setminus F\in J\setminus I$, so $\vert E\setminus F\vert=\infty$. 
Note that $(E\setminus F)+1=(E+1)\cap F\in I$ (in particular if $I$ is shift-invariant, that is $B\in I$ if and only if $B+1\in I$, then we immediately get the contradiction, since $E\setminus F\notin I$ and $(E\setminus F)+1\in I$). Define $G=(E\setminus F)\cup ((E\setminus F)+1)$, then $G\in J\setminus I$. Fix $y=\sum_{n\in G}x_n$. Since $y_n>2\sum_{k=n+1}^{\infty}y_k$, then the equality $y=\sum_{n\in H}x_n$ holds if and only if $H=G$. Hence $y\in A_{J}(x_n)\setminus A_{I}(x_n)$. We proved that the series $\sum_{n=1}^{\infty}x_n$ satisfies the condition $(3)$ from Proposition \ref{rowne}. 

\end{example}

Intersection of two ideals is also an ideal. The following proposition is connected with such ideal. 
\begin{proposition}\label{przekrojidealow}
Assume that $I$ and $J$ are ideals. Let $\sum_{n=1}^{\infty}x_n$ be a series such that its associated function $f$ is injective. 
Then $A_{I\cap J}(x_n)=A_{I}(x_n)\cap A_{J}(x_n)$.
\begin{proof}
It is clear that $A_{I\cap J}(x_n)\subset A_{I}(x_n)\cap A_{J}(x_n)$. Let $x\in A_{I}(x_n)\cap A_{J}(x_n)$. Then $x=\sum_{n\in A} x_n$ and $x=\sum_{n\in B} x_n$ for some $A\in I,B\in J$. From the assumption we have $A=B$, so $A\in I\cap J$. Hence $x\in A_{I\cap J}(x_n)$.
\end{proof}
\end{proposition}
We can strengthen Proposition \ref{przekrojidealow} by modifing its assumptions: 
\begin{proposition}\label{rownowaznyprzekroj}
Let $\sum_{n=1}^{\infty}x_n$ be an absolutely convergent series. If the associated function $f$ is injective on $W=\{0,1\}^{\mathbb{N}}\setminus\{\chi_A : \vert A\vert<\infty  \ \text{or} \ \vert \mathbb{N}\setminus A\vert<\infty\}$, then  for every ideals $I,J$ we have $A_{I\cap J}(x_n)=A_{I}(x_n)\cap A_{J}(x_n)$. 
\begin{proof}
Let take two ideals $I,J$ and fix $x\in A_{I}(x_n)\cap A_{J}(x_n)$, that is  $x=\sum_{n\in A} x_n=\sum_{n\in B} x_n$ for $A\in I$ and $B\in J$.  If $A=B$, then  $x\in A_{I\cap J}(x_n)$. Suppose that $A\neq B$. Since $A,B$ cannot be cofinite as elements of ideals, we get $A\in Fin\subset I\cap J$ or $B\in Fin\subset I\cap J$. Hence $x\in A_{I\cap J}(x_n)$.
\end{proof}
\end{proposition}
\begin{example}
Let $x_n=\frac{1}{2^n}$ for each $n\in\mathbb{N}$. Since $\sum_{n=1}^{\infty}x_n$ satisfies the assumptions of Proposition \ref{rownowaznyprzekroj}, we obtain $A_{I\cap J}(x_n)=A_{I}(x_n)\cap A_{J}(x_n)$ for all ideals $I,J$. 
\end{example}

\section{Open problems}


In Remark \ref{bezwlasnoscibairea} for a maximal ideal we constructed a sequence for which $A_{I}(x_n)$ does not have the Baire property. In particular it means that $A_{I}(x_n)$ is neither a meager nor a comeager set. Other examples lead us to state the following:
\begin{problem}
Assume that $A_{I}(x_n)$ has the Baire property. Is it true that  $A_{I}(x_n)$ is meager or comeager ?
\end{problem} 

Section 6 was dedicated to some equalities and inclusions connected with ideally supported achievement set. We considered the following conditions: 
\begin{enumerate}
\item  for all ideals $I,J$ we have $A_{I\cap J}(x_n)=A_{I}(x_n)\cap A_{J}(x_n)$; 
\item for every ideals $I\neq J$ we have $A_{I}(x_n)\neq A_{J}(x_n)$;
\item for every ideals $I\subsetneq J$ we have $A_{I}(x_n)\subsetneq A_{J}(x_n)$.
\end{enumerate}
We showed that if the associated function $f$ of the series $\sum_{n=1}^{\infty}x_n$ is incjective, then all three above conditions are satisfied. Moreover we presented examples, which show that the above implication cannot be reversed for all three conditions. 
\begin{problem}
Characterize classes of series, which satisfy the above conditions.
\end{problem} 

\textbf{Acknowledgment.} I would like to thank my PhD's supervisor prof. Szymon G\l \c{a}b for a very careful analysis and fruitful discussions on this article.

\end{document}